\documentclass[10pt]{amsart}

\usepackage{graphicx}

\usepackage{hyperref}
\usepackage{amsmath,amssymb,cite}
\usepackage{mathrsfs}
\usepackage{amsthm}
\newtheorem{thm}{Theorem}

\newtheorem{lem}{Lemma}
\newtheorem{ex}{Example}
\newtheorem{cor}{Corollary}

\theoremstyle{definition}

\newtheorem{prob}{Open Problem}

\numberwithin{equation}{section}

\theoremstyle{remark}
\newtheorem{rem}{Remark}
\newcommand{\R}{{\mathbb R}}                    
\newcommand{\N}{{\mathbb N}}                    
\newcommand{\II}{{\mathbf I}}                    

\newcommand{\CC}{{\rm C}}

\newcommand{\id}{\operatorname{id}}             

\newcommand{\diam}{\operatorname{diam}}         

\newcommand{\dd}{\downarrow\nobreak\!\!}         

\begin{document}

\vspace{2mm} \baselineskip 15pt
\renewcommand{\baselinestretch}{1.10}
\parindent=16pt  \parskip=2mm
\rm\normalsize\rm

\title[Function space of transitive maps]{The topological structure of function space of transitive maps}
\author{Zhaorong He, Jian Li and Zhongqiang Yang${}^{\mathbf{*}}$}
\address{Department of Mathematics,
Shantou University, Shantou, Guangdong, 515063, China
P.R.}\email{[Z. He]17zrhe@stu.edu.cn,\ [J.Li]lijian09@mail.ustc.edu.cn\ and \newline \hspace*{2.63cm} [Z.Yang]zqyang@stu.edu.cn}
\begin{abstract} Let $C(\mathbf I)$ be the set of all continuous self-maps from ${\mathbf I}=[0,1]$  with the topology of uniformly convergence. A map $f\in C({\mathbf I})$ is called a transitive map
if for every pair of non-empty open sets $U,V$ in $\mathbf{I}$, there exists a positive integer  $n$ such that $U\cap f^{-n}(V)\not=\emptyset.$ We note $T(\mathbf{I})$ and $\overline{T(\mathbf{I})}$ to be the sets of all transitive maps and its closure in the space $C(\mathbf I)$. In this paper, we show that $T(\mathbf{I})$ and $\overline{T(\mathbf{I})}$ are homeomorphic to the separable Hilbert space $\ell_2$. \end{abstract}

\thanks{This paper was supported by the NNSF of China
	(Nos. 11971287, 11471202 and 11771264)
	and NSF of Guangdong Province(2018B030306024).\\  ${}^{\mathbf{*}}$ corresponding author.}

\maketitle

\section{Introduction and main results}
In this paper, we investigate the topological structure of function space of transitive maps from a
closed interval to itself.

Investigating the topological properties of function spaces is an important subject in many branches of mathematics. 
See, for example,  \cite{McCoy-Ntantu-book, McCoy-Ntantu-1992, van-book-2001, Yang-2012} and \cite[Chapter 7]{Kelly-book-1952}. 
In particular, various versions of Ascoli Theorem  give necessary and sufficient conditions 
to decide whether subsets of function spaces with various types of topologies are compact, see e.g. \cite [Charter 7]{Kelly-book-1952}. 
With the development of infinite-dimensional topology,  it appears  some structural characteristics of these function spaces.
For example, by the Kadec's theorem  in \cite{Kadec}, we know that the space of the real-valued continuous functions of a compact metric space with the  topology of uniformly convergence is homeomorphic to the separable Hilbert space $\ell_2$. 
In  \cite{Dobrowolski-Marciszewski-1991}, Dobrowolski et al showed that  the space of real-valued continuous functions of a countable non-discrete metric space with the topology of pointwise convergence is homeomorphic to the subspace $c_0=\{(x_n)\in \R^\infty:\lim_{n\to\infty} x_n=0\}$ of the countable product $\R^\infty$ of real lines. 
In 2005-2017, the third named author of the present paper and his coauthors obtained structural characteristics of  spaces 
of continuous functions $\dd C_F(X)$ from a k-space $X$ to $\II=[0,1]$ with the Fell topology of hypograph, 
see \cite{Yang-2012, Yang-2005, Yang-2006, Yang-Zhou-2007, Yang-Wu-2009, Yang-Zhang-2012, Yang-Hu-2013, Yang-Yan-2014, Yang-Zheng-Chen,Yang-Chen-Zheng, Yang-Yang-Zhen}.
For example, $\dd C_F(X)$ is homeomorphic to $c_0$ if $\dd C_F(X)$ is metrizable and the set of isolated points in $X$ is not dense.  Spaces of fuzzy numbers with various types of topologies can be also regard as function spaces.
In \cite{Yang-Zhang-2009, Yang-Zeng-2019, Zhang-2013},  the topological structure of spaces of fuzzy numbers was investigated. For example, some of them are homeomorphic to $\ell_2$ and others are homeomorphic to the pseudo-boundary of the Hilbert cube. 

For a topological dynamical system, we mean a pair $(X,f)$,
where $X$ is a compact metric space and $f:X\to X$ is continuous.
Let $C(X)$ be the function space of all continuous maps from $X$ to $X$ with the supremum metric.
Moreover, many maps satisfying some topological or dynamical properties are defined, such as, all homeomorphisms, all transitive maps etc. A lot of papers  discuss dynamical properties of those maps (that is,  properties of the sequence $\{f^n\}$ of functions), while there exists only a few papers to investigate  topological
properties of function spaces consisting of maps with some specific dynamical properties.  In \cite{Grinc},  Grinc et al, using Sharkovsky order, defined some subspaces of the function spaces of self-maps on compact interval with the supremum metric, and proved that some of them are of second category and others are of first category.
In \cite{Kolyada-2151}, Kolyda et al considered the space of  all transitive self-maps on a compact interval, all piecewise monotone transitive self-maps and all piecewise linear transitive self-maps. It was proved in \cite{Kolyada-2151} that those spaces are contractible and uniformly locally arcwise connected. Moreover, the authors in \cite{Kolyada-2151} said: ``Investigating the topological properties of spaces of maps that can be described in dynamical terms is in a sense the opposite idea (of topological dynamics). Therefore we propose to call this area {\bf dynamical topology}." In \cite{Kolyada-137}, the authors continued to discuss those spaces defined in \cite{Kolyada-2151}. They showed that some loops which are  not contractible in some of those spaces can be contractible in slightly larger spaces.

As well-known, almost all function spaces, including function spaces of maps that can be described in dynamical terms, are infinite-dimensional.  Hence, it is natural to give some structural characteristics of
those spaces using  tools of infinite-dimensional topology. In \cite{Fan-2018}, Fan et al showed that some subspaces of the space of continuous  self-maps on  a compact interval with the supremum metric related with the topological entropy are homeomorphic to  $\ell_2$. We know that
the transitive property
is an important concept in topological dynamics.
 In the present paper, we will show  the space of transitive self-maps of
compact interval with the supremum metric is homeomorphic to  $\ell_2$. This is a more precise result than ones in \cite{Kolyada-2151}.

For $\II=[0,1]$, let $C(\II)$ be the set of continuous maps from $\II$ to itself and endow $C(\II)$ with the topology of uniformly convergence. A map $f\in C(\II)$ is called a {\bf transitive map}
if for every pair of non-empty open sets $U,V$ in $\mathbf{I}$, there exists a positive integer $n$ such that $U\cap f^{-n}(V)\not=\emptyset.$ It is not hard to verify that $f\in C(\II)$ is a transitive map
if and only if the orbit $\{f^n(x):n=1,2,\cdots\}$ of some point $x\in \II$ is dense in $\II$ if and only if  every non-empty open set $U\subset \II$, $\bigcup_{n=1}^\infty f^n(U)$ is dense in $\II$. Hence, all transitive maps are surjective. The set of all transitive maps from $\II$ to itself is denoted by $T(\mathbf{I})$. Moreover, we denote its closure in the space $C(\mathbf I)$ by $\overline{T(\mathbf{I})}$. Let $\ell_2$ be the separable Hilbert space.
Our main result in this paper is as follows.

\bigskip
\noindent{\bf Main Theorem.} {\it Both spaces $T(\II)$ and $\overline{T(\mathbf{I})}$ are homeomorphic to the Hilbert space $\ell_2$.}

Moreover, we have the following corollaries.
\begin{cor}\label{Z-set}
There exists a homeomorphism $h:C(\II)\to \II\times \ell_2$ such that $h(\overline{T(\II)})=\{0\}\times \ell_2.$
\end{cor}

\begin{cor}\label{near-homeomorphism}
For every open cover $\mathcal{U}$ of $\overline{T(\II)}$, there exists a homeomorphism $h:T(\II)\to \overline{T(\II)}$ such that $(\id_{T(\II)},h)\prec\mathcal{U},$ that is, for every $f\in T(\II)$, there
exists $U\in\mathcal{U}$ such that $f,h(f)\in U$.
\end{cor}

The paper is organized as follows. In Section 2, we recall some basic notions and results
which we will use in the paper. In particular, we will give the concept of the box map defined in  \cite{Kolyada-2151}.
In Section 3, a property on  continuous extension is given.  Main Theorem will be proved in Section 4. In the last section, some remarks, examples and open problems are put.

\section{Preliminaries}

In this paper, all subsets of the set of real numbers $\R$ are thought to be subspaces of $\R$  with the usual topology. In particular, $\II=[0,1]$ is with the usual topology and $\N$ is the set of natural numbers with the discrete topology. When we say ``piecewise", we mean that there are finitely many pieces. For two compact sets $A,B\subset\R$, let $C(A,B)$ be the set of all continuous maps from $A$ to $B$.  Then $C(\II)=C(\II,\II)$.  When $C(A,B)$ is thought to be a metric space, it is always  considered with the supremum metric:
$$d(f,g)=\sup\{|f(x)-g(x)|:x\in A\},\ \ \ \ \ f,g\in C(A,B).$$
For two spaces $X$ and $Y$ and their two subspaces $A$ and $B$, respectively, $(X,A)\approx (Y,B)$ means that there exists a homeomorphism $h:X\to Y$ such that $h(A)=B$. Similarly, we can define $X\approx Y, ~(X,A_1,A_2)\approx (Y,B_1,B_2),
$ etc.

At first, we give some concepts and results in infinite-dimensional topology. For more information on them, we refer reader to \cite{Banakh-book,van-book-1989,van-book-2001,Sakai-book-2013,Sakai-book-2019}.

A closed set $C$ in a topological space $X$ is called a {\bf Z-set} if for every open cover $\mathcal{U}$, there exists a continuous map
$f:X\to X\setminus C$ such that $f$ and $\id_X$ are $\mathcal{U}$-close, that is for every $x\in X$
there exists $U\in\mathcal{U}$ with $x,f(x)\in U$.
By Theorem of Z-set Unknotting (see, for example, \cite[Theorem 2.9.7]{Sakai-book-2019}), we have  that $(L,A)\approx (\ell_2\times \II, \ell_2\times\{0\})$ if $A\approx L\approx \ell_2$ and $A$ is a Z-set in $L$. A space $X$ is called a {\bf Z$_
\sigma$-space} if $X$ is a countable union of Z-sets in $X$. Trivially, no topologically completed space  is a Z$_
\sigma$-space.

 A metrizable space $X$ is called an {\bf absolute  (neighborhood) retract} (in brief, {\bf A(N)R}) if for every metrizable space $Z$ which includes $X$ as a closed subspace, there exists a retraction $r:Z\to X$ (respectively, a retraction $r:U\to X$ from a neighborhood $U$ of $X$ in $Z$). It is well-known that a metrizable space is an AR if and only if it is a contractible ANR, see e.g.\ \cite[Corollary 6.2.9]{Sakai-book-2013}.
 A subspace $Y$ of $X$ is called {\bf homotopy dense} if there exists a homotopy
$H:X\times\II\to X$ such that $H_0=\id_X$ and $H_t(X)\subset Y$ for every $t\in (0,1]$. The complement of a homotopy dense set is call {\bf homotopy negligible}. Note that a set in a space may be both homotopy dense and homotopy negligible.
 But if a closed set $C$ is included in a homotopy negligible set, then $C$ is a Z-set. It have been proved that, for a homotopy dense subspace $Y$ of  a metric space $X$, $Y$ is an A(N)R if and only if $X$  is an A(N)R, see e.g.\ \cite[Corollary 6.6.7]{Sakai-book-2013} or \cite[Exercise 1.2.16]{Banakh-book}

 It is sometime difficult to verify that a metrizable space is an A(N)R. In the present paper, we will use the following result. A separable metric space $(X,d)$ is an ANR if  the following statement holds:\footnote{In \cite[Theorem 5.2.1]{van-book-1989}, a necessary and sufficient condition using open cover for a separable metrizable space being an ANR is given. It is not hard to verify that Statement E is stronger than this condition. The space $(0,1)\cup (1,2)$ with the usual metric is an ANR but does not satisfy Statement E. }

\hspace*{-0.55cm}Statement E\hspace*{0.4cm} \begin{minipage}{0.75\linewidth}
  there exists $C>1$ such that for every  countable and locally finite simplicial complex $K$ and its subcomplex
   $L \supset K^{(0)} $,  for every continuous map   $\phi:|L|\to X$ there exists a continuous extension $\Phi:|K|\to X$ of $\phi$ such that
\begin{equation}\label{leq}
\diam\Phi(\sigma)\leq C\diam \phi(|L|\cap\sigma)
\end{equation}
for every $\sigma\in K$, where $\diam A$ is the diameter of a subset $A$ of  the metric space $(X,d)$.
\end{minipage}

\hspace*{-0.55cm}Statement E  is also used to verify a subset of a space to be homolopy dense. In \cite{Sakai-2000}, it is shown that $X$ is homotopy dense in $Y$ if  $X$ is  a dense subspace of a separable metrizable $Y$ and satisfies  Statement E.

A metric space $(X,d)$ is said to have the {\bf strongly discrete approaching property} (in brief, {\bf SDAP}) if for every continuous map $\varepsilon:X\to (0,1)$, every compact metric space $K$ and every continuous map $f:K\times \N\to X$, there exists a continuous map $g:K\times\N\to X$ such that
  $\{g(K\times\{n\}):n\in\N\}$ is discrete and $d(f(k,n),g(k,n))<\varepsilon(f(k,n))$ for every $(k,n)\in K\times \N$. In \cite[1.3.1 Proposition]{Banakh-book}, it is shown that  the above ''discrete" can be replaced by ''locally finite". Moreover, every homotopy dense subspace of  an ANR with SDAP has SDAP, see e.g.\ \cite[Exercise 1.3.4]{Banakh-book}.

We recall the  Toru\'{n}czyk's Characterization Theorem of Separable Hilbert Space:
A separable meterizable space $X$ is homeomorphic the separable Hilbert space  $\ell_2$ (an $\ell_2$-manifold) if and only if $X$ is a topologically completed A(N)R with SDAP,  see e.g.\ \cite[1.1.14 (Characterization Theorem)]{Banakh-book}.

Secondly, we introduce a family of homotopies $\{H^\gamma:C(\II)\times\II\to C(\II):\gamma\geq 20\}$ which defined in \cite{Kolyada-2151}. \footnote{In \cite{Kolyada-2151}, $\gamma$ and $H^\gamma(f,t)$ in the present paper were denoted by $a_s$ and $g_{f,t}$, respectively. Although $a_s$ can be thought to be as a variable, it  was fixed in most part of \cite{Kolyada-2151}. In the present paper, it is important to consider  $\gamma$ as a variable.} defined in \cite{Kolyada-2151}.
 Define a subspace $\Lambda$ of $\II^4\times [20,+\infty)$ as follows
$$\Lambda=\{(a_l,a_r,a_b,a_t,\gamma)\in\II^4\times [20,+\infty):a_b<a_t,a_l,a_r\in [a_b,a_t]\}.$$
For every non-degenerate closed interval $K=[a_0,a_1]$ and $\lambda=(a_l,a_r,a_b,a_t,\gamma)\in\Lambda$, the authors in \cite{Kolyada-2151} defined a continuous surjection
$\xi_\lambda:K\to [a_b,a_t]$, which was called a {\bf box map}, such that  $\xi_\lambda$ is piecewise linear with constant slope $\frac{\gamma(a_t-a_b)}{a_1-a_0}$ and $\xi_\lambda(a_0)=a_l,\xi_\lambda(a_1)=a_r$. The following figure illustrates the definition of box map $\xi_\lambda$, where the meeting point $m$ can be chosen to be the fifth decreasing lap from the left.

\begin{figure}[htbp!]
   \centering
   \includegraphics[scale=0.5]{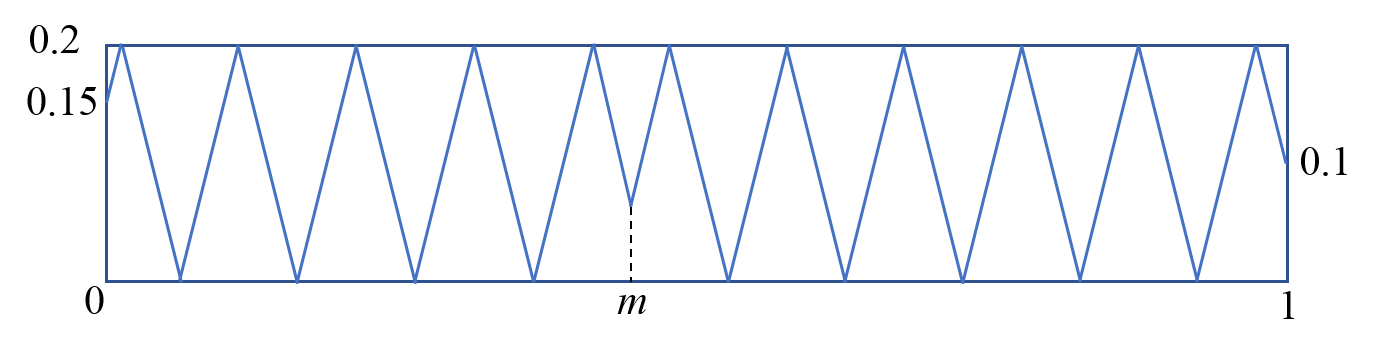}
   \caption{$K=[0,1]$, $a_l=0.15$, $a_r=0.1$, $a_b=0$, $a_t=0.2$, $\gamma=20$}
   \label{fig:box-map}
 \end{figure}

Using the box map, Kolyada et al in \cite{Kolyada-2151} constructed a homotopy\footnote{In \cite{Kolyada-2151}, this homotopy was defined in a subspace of $C(\II)\times \II$. But it is also valid for the space $C(\II)\times\II.$} $H^\gamma:C(\II)\times\II\to C(\II)$ for every $\gamma\ge 20$ as follows. Let $H^\gamma_0=\id$ and, for $f\in C(\II)$ and $t\in (0,1]$, $s=s(t)$ be the largest non-negative integer such that $st<1$. We can obtained $s+1$ closed intervals:
$$I_i^t=[(i-1)t,it],\ i=1,2,\cdots,s,\ I_{s+1}^t=[st,1].$$
In particular, if $t=1$, then $s=0$ and we obtain one closed interval $I_1^1=[0,1]$. For $i=1,2,\cdots,s(t)+1$, let
$\alpha_i^{f,t}=\max\{|I_i^t|, |f(I_i^t)|\}$, where $|J|$ is the length of a closed interval $J$. Moreover, let
$$a^{i,f,t}_{b}=\max\{0,\min f(I_i^t)-4\alpha_i^{f,t}\};\ a^{i,f,t}_t=\min\{1,\max f(I_i^t)+4\alpha_i^{f,t}\};$$$$J_i^{f,t}=[a^{i,f,t}_{b},a^{i,f,t}_{t}];\ a^{i,f,t}_l=f(\min I_i^t);\ a^{i,f,t}_r=f(\max I_i^t).$$
It is not hard to verify that
\begin{equation}\label{1}|I_i^t|\leq |J_i^{f,t}|.\end{equation}
Trivially, $\lambda_{i,f}^\gamma=(a^{i,f,t}_l,a^{i,f,t}_r,a^{i,f,t}_b,a^{i,f,t}_t,\gamma)\in \Lambda.$ Therefore, we can define a box map $\xi_{\lambda_{i,f}^\gamma}\in C(I_i^t,J_i^{f,t})$.
Moreover, using them, we can define

$$H^\gamma(f,t)=\bigcup_{i=1}^{s+1}\xi_{\lambda_{i,f}^\gamma}\in C(\II).$$
Kolyada et al in \cite{Kolyada-2151} proved that $H^\gamma:C(\II)\times\II\to C(\II)$ is continuous for every $\gamma\geq 20$, even the map also induced a continuous map from $C(\II)\times\II\times [20,+\infty)$ to $C(\II)$ when we think $\gamma$ to be a
variable in $[20,+\infty)$. Moreover, Kolyada et al noted that it is also true if we replace $C(\II)$ by $T(\II)$. Using this result, it is trivial that $T(\II)$
is {\bf contractible}, that is, the identity map $\id_{T(\II)}$ and a constant map are homotopic in $T(\II)$.

\section{Continuous Extensions}
Let $(X,d)$ be a metric space. We introduce Statement SC as follows:

\hspace*{-0.55cm}Statement SC\hspace*{0.4cm} \begin{minipage}{0.75\linewidth}
   For every $\varepsilon>0$, every continuous map $\phi:\partial\sigma\to X$ from the boundary $\partial\sigma$ of a simplex $\sigma$ to $X$ and every compact set $C\supset \phi(\partial\sigma)$ in $X$,  there exists a continuous extension $\Phi:\sigma\to X$ of $\phi$ such that
\begin{equation}\label{leq8}  \diam(\Phi(\sigma)\cup A)\leq (1+\varepsilon)\diam A\end{equation}
if $ \phi(\partial\sigma)\subset A\subset C.$
\end{minipage}

 The main purpose of this section is to show the following theorem.
 \begin{thm}\label{extension 2} Every metric space $(X,d)$ satisfies Statement E if it satisfies Statement SC.
 \end{thm}
\begin{proof} Let $K$ be a countable and locally finite simplicial complex and $L \supset K^{(0)} $ a subcomplex of $K$. For every continuous map   $\phi:|L|\to X$, we  show that there exists a continuous extension $\Phi:|K|\to X$ of $\phi$ such that
\begin{equation}
\diam\Phi(\sigma)\leq 2\diam \phi(|L|\cap\sigma)
\end{equation}
for every $\sigma\in K$. Hence $(X,d)$ satisfies Statement (E).

Choose $\varepsilon_i>0$ such that $$\prod_{i=1}^\infty(1+\varepsilon_i)\leq 2.$$
Let $K\setminus L=\{\sigma_i\}$ and $C_i=\bigcup\{\tau\in K:\sigma_i\prec \tau\}$. Then $C_i$ is compact since $K$ is locally finite.

Using Statement SC,  we will inductively define a sequence \{$\Phi_n:K_n\to X\}_{n=0}^\infty$ of  continuous maps, where $K_n=|L|\cup |K^{(n)}|$, such that
\begin{enumerate}
\renewcommand{\labelenumi}{(\roman{enumi})}
\item $\Phi_0=\phi;$

\item $\Phi_{n}|K_{n-1}=\Phi_{n-1};$

\item $\diam (\Phi_n(\sigma_i)\cup A)\leq (1+\varepsilon_i)\diam A$ if $$\Phi_{n-1}(\partial\sigma_i)\subset A\subset \Phi_{n-1}(C_i\cap K_{n-1})\cup\Phi_n(\bigcup_{j<i}\{\sigma_j:\dim \sigma_j=n\})$$ for every $\sigma_i\in K\setminus L$ with $\dim \sigma_i=n$.

\end{enumerate}
In fact, let $\Phi_0=\phi$ and $\Phi_{n-1}$ has been defined and satisfies the inductive assumptions.
Let $$\{\sigma_{i_j}:j=1,2,\cdots\}=\{\sigma_i:\dim \sigma_i=n\},$$
where $i_1<i_2<\cdots.$ For a fixed $j_0=1,2,\cdots$, suppose that,  for every $j<j_0$, $\Phi_n|\sigma_{i_j}$ has been defined and satisfies
$\Phi_{n}|\partial\sigma_{i_j}=\Phi_{n-1}|\partial\sigma_{i_j}$ and
$$\diam (\Phi_n(\sigma_{i_j})\cup A)\leq (1+\varepsilon_{i_j})\diam A$$
if
$$\Phi_{n-1}(\partial\sigma_{i_j})\subset A\subset \Phi_{n-1}(C_{i_j}\cap K_{n-1})\cup\Phi_n(\bigcup_{j'<j}\{\sigma_{i_j'}\}).$$
Then $$C=\Phi_{n-1}(C_{i_{j_0}}\cap K_{n-1})\cup\Phi_n(\bigcup_{j<j_0}\{\sigma_{i_j}\})$$
is compact in $X$. Hence, for $\Phi_{n-1}|\partial\sigma_{i_j}$ and $\varepsilon=\varepsilon_{i_j}$, since $(X,d)$ satisfies Statement SC, there exists a continuous extension $\Phi_n|\sigma_{i_j}:\sigma_{i_j}\to X$ such that
$$\diam(\Phi_n(\sigma_{i_j})\cup A)\leq (1+\varepsilon_{i_j})\diam A$$
if $\Phi_{n-1}|\partial\sigma_{i_j}\subset A\subset C.$ Then it is not hard to verify that $\Phi_n=\Phi_{n-1}\cup\bigcup_{j=1}^\infty \Phi_n|\sigma_{i_j}:|K_n|\to X$ satisfies  the inductive assumptions (i)-(iii).

Trivially, $\Phi=\bigcup_{n=1}^\infty \Phi_n:|K|\to X$ is a continuous extension of $\phi:|L|\to X$. It is remainder to  check that   Formula (\ref{leq}) holds for every $\sigma\in K$. We only consider the case that $\sigma=\sigma_i\in K\setminus L$ with $\dim \sigma_i=n$. By (iii) for the case $A=\Phi_{n-1}(\partial\sigma_i)$, we have
\begin{equation}\label{leq 2}\diam \Phi(\sigma_i)\leq (1+\varepsilon_i)\diam \Phi(\partial\sigma_i).\end{equation}
In the finite set $N=\{j:\sigma_j\precneqq\sigma_i\}$, where $\sigma_j\precneqq\sigma_i$ means that $\sigma_j$ is a proper face of $\sigma_i$, we define a linear order $\vartriangleleft$  as  $j_1\vartriangleleft j_2$ if and only if either $\dim \sigma_{j_1}>\dim\sigma_{j_2}$ or $\dim \sigma_{j_1}=\dim\sigma_{j_2}$
and $j_1>j_2$. Let $N=\{j_1\vartriangleleft j_2\vartriangleleft\cdots\vartriangleleft j_l\}.$
 Note that $$\Phi(\partial\sigma_i)=\Phi(\sigma_{j_1})\cup\Phi(\sigma_{j_2})\cup\cdots\cup\Phi(\sigma_{j_l})\cup\phi(|L|\cap \sigma_i).$$
 Now we give an estimation for $\diam \Phi(\partial\sigma_i)$. At first,
$$\Phi(\partial\sigma_{j_1})\subset \Phi(\sigma_{j_2})\cup\cdots\cup\Phi(\sigma_{j_l})\cup\phi(|L|\cap \sigma_i)$$$$\subset \Phi(C_{j_1}\cap K_{\dim \sigma_{j_1}-1})\cup\Phi(\bigcup_{j<j_1}\{\sigma_j:\dim \sigma_j=\dim \sigma_{j_1}\}).$$
It follows from (iii) and Formula (\ref{leq 2}) that
$$\diam \Phi(\sigma_i)\leq (1+\varepsilon_i)(1+\varepsilon_{j_1})\diam\Phi(\sigma_{j_2})\cup\cdots\cup\Phi(\sigma_{j_l})\cup\phi(|L|\cap \sigma_i).$$
Continuously this, we have

$$ \diam\Phi(\sigma_i)\leq (1+\varepsilon_i)\prod_{m=1}^l(1+\varepsilon_{j_m})\diam \phi(|L|\cap \sigma_i)\leq 2  \phi(|L|\cap \sigma_i).$$
\end{proof}

\section{Proof of the Main Theorem}
At first, we show that $T(\II)$ satisfies Statement E.
\begin{lem}\label{compact}
For every  compact set $K$ in $C(\II)$ and $\varepsilon>0$, there exists $t_0>0$ such that for every $t\in (0,t_0)$, $i\leq s(t)+1$ and $A\subset K$,
 $$a_{t}^{i,A,t}-a_{b}^{i,A,t}\leq\diam A+\varepsilon,$$
 where $a_{t}^{i,A,t}=\sup\{a_{t}^{i,f,t}:f\in A\}, a_{b}^{i,A,t}=\inf\{a_{b}^{i,f,t}:f\in A\}.$
\end{lem}
\begin{proof}  Note that $K$  is uniformly equicontinuous, that is, for every $\eta>0$, there exists $\delta>0$ such that
$$|x_1-x_2|<\delta\ \mbox{and}\ f\in K \ \mbox{imply}  \ |f(x_1)-f(x_2)|<\eta.$$
Hence, for  $\eta=\frac{\varepsilon}{10}>0$, there exists $t_0>0$ such that
$$\alpha_i^{f,t}=\max\{|I_i^{t}|,|f(I_i^{t})|\}<\eta$$
for every $t\in (0,t_0]$ and $f\in K$. Thus, for $f,\widetilde{f}\in A$ and $i\le s(t)+1$, choose $x_0\in I_i^t$, then
$$
\begin{array}{ll}
&a_{t}^{i,f,t}-a_{b}^{i,\widetilde{f},t}\\
\leq & \max f(I_i^{t})+4\alpha_i^{f,t}-\min \widetilde{f}(I_i^{t})+4\alpha_i^{\widetilde{f},t}\\
\leq & \eta+f(x_0)+8\eta -\widetilde{f}(x_0)+\eta\\
\leq & \diam A+\varepsilon.\\
\end{array}$$
Therefore, $$a_{t}^{i,A,t}-a_{b}^{i,A,t}\leq\diam A+\varepsilon$$
 for every $t\in (0,t_0)$, $i\leq s(t)+1$ and $A\subset K$.
\end{proof}

\begin{lem}\label{extension} The space $T(\II)$ satisfies Statement SC.

\end{lem}
\begin{proof} Let $\varepsilon>0$, $\phi:\partial\sigma\to T(\II)$ a continuous map from the boundary $\partial\sigma$ of a simplex $\sigma$ and $C\supset \phi(\partial\sigma)$ a compact set in $T(\II)$. We show that  there exists a continuous extension $\Phi:\sigma\to T(X)$ of $\phi$ such that
\begin{equation}\label{leq1}  \diam(\Phi(\sigma)\cup A)\leq (1+\varepsilon)\diam A\ \ \mbox{for every}\ \phi(\partial\sigma)\subset A\subset C.\end{equation}
This conclusion is trivial if $\diam\phi(\partial\sigma)=0.$ Hence, suppose  $\diam\phi(\partial\sigma)>0.$
Using Lemma \ref{compact}, there exists $t_0\in (0,1)$ such that
\begin{enumerate}
\renewcommand{\labelenumi}{(\roman{enumi})}
\item $d(H^{20}(\phi(x),t ),\phi(x))<\frac{\varepsilon\diam\phi(\partial\sigma)}{4}$;

\item $a_{t}^{i,A,t}-a_{b}^{i,A,t}\leq\diam A+\frac{\varepsilon\diam\phi(\partial\sigma)}{2}\ \ \mbox{for every}\ \phi(\partial\sigma)\subset A\subset K,$

\end{enumerate}
for every $t\in [0,t_0]$, $x\in \partial\sigma$ and $i\leq s(t)+1$. Moreover, for $x\in \partial\sigma$ and $i\leq s(t_0)+1$, let
$$\lambda^x_i=(a_l^{i,\phi(x),t_0},a_r^{i,\phi(x),t_0},a_b^{i,\phi(x),t_0},a_t^{i,\phi(x),t_0},20)\in\Lambda~~~~\mbox{and}$$
$$\lambda_i=(a_{t}^{i,\phi(\partial\sigma),t_0},a_{b}^{i,\phi(\partial\sigma),t_0},a_{b}^{i,\phi(\partial\sigma),t_0},a_{b}^{i,\phi(\partial\sigma),t_0},20)\in \Lambda.$$
Using them, we can define a linear map $l_i^x:[t_0,1]\to \Lambda$  such that $l_i^x(t_0)=\lambda^x_i$ and $l_i^x(1)=\lambda_i$. For every $i\leq s(t_0)+1, x\in \partial\sigma$ and $t\in [t_0,1]$, $\xi_{l_i^x(t)}$
defines a box map from $I_i^{t_0}$ to $\II$. Trivially, $\{\xi_{l_i^x(t)}:i\leq s(t)+1\}$ can be joined into a map $G(x,t)\in C(\II)$. Note that  $G(x,1)=\cup_{i\leq s(t_0)+1}\xi_{\lambda_i}$ is independent  of  $x\in \partial\sigma$. Since $G(x,t_0)=H^{20}(\phi(x),t_0)$ for every $x\in \partial\sigma$, we can define $\Phi: \sigma\to C(\II)$ as follows:
 \begin{equation*}
\Phi((1-t)x+tb) = \begin{cases}
H^{20}(\phi(x),t) & ~t\in [0,t_0], \\
G(x,t)&~t\in [t_0,1]
\end{cases}
\end{equation*}
for $(1-t)x+tb\in \sigma$, where $b$ is the barycenter of $\sigma$ and $x\in \partial\sigma$. What follows is to verify that $\Phi$ satisfies all requirements.

At first, trivially, $\Phi$ is a continuous extension of $\phi$.

Secondly, by the proof of \cite[Lemma 2.3]{Kolyada-2151} and \cite[Remark 2.4]{Kolyada-2151}, we know that $\Phi(\sigma)\subset T(\II).$

Third,  we show that Formula (\ref{leq1}) holds.
Fix $A$. We write  $\sigma=\sigma_1\cup \sigma_2$, where $$\sigma_1=\{(1-t)x+tb:(x,t)\in \partial\sigma\times [0,t_0]\}\ \ \ \mbox{and}$$ $$\sigma_2=\{(1-t)x+tb:(x,t)\in \partial\sigma\times [t_0,1]\}.$$
Then, using (i) and (ii), respectively,  we have
\begin{equation}\label{leq-1}
\diam(\Phi(\sigma_1)\cup A)\leq \diam A+\frac{\varepsilon}{2}\diam A\  \mbox{and}
\end{equation}
\begin{equation}\label{leq-2}
\diam(\Phi(\sigma_2)\cup A)\leq \diam A+\frac{\varepsilon}{2}\diam A.
\end{equation}
In fact, Formula (\ref{leq-1}) follows (i). To verify Formula (\ref{leq-2}) we note that (ii) and,  for every $(x,t)\in\sigma\times [t_0,1]$, $i\leq s(t)+1 $ and $f\in A$, the graphs of $f|I_i^{t}$ and $\Phi(x)|I_i^{t}$ are included in $I_i^{t}\times [a_{b}^{i,A,t},a_{t}^{i,A,t}]$.

For any $t\in [t_0,1], t'\in [0,t_0]$ and $x.x'\in\partial\sigma$, using (i) and Formula (\ref{leq-2}) for $A=\phi(\partial\sigma)$, we have
$$\begin{array}{rl}
& d(\Phi((1-t)x+tb),\Phi((1-t')x'+t'b))\\
\leq & d(\Phi((1-t)x+tb),\Phi((1-t_0)x'+t_0b))\\
&+\ d(\Phi((1-t_0)x'+t_0b),\Phi((1-t')x'+t'b)\\
\leq & \diam \phi(\partial\sigma)+\frac{\varepsilon\diam \phi(\partial\sigma)}{2}+\frac{\varepsilon\diam\phi(\partial\sigma)}{2}\\
=&(1+\varepsilon)\diam\phi(\partial\sigma)\\
\leq  & (1+\varepsilon)\diam A.
\end{array}
$$
It follows from Formulas (\ref{leq-1}) and (\ref{leq-2}) that Formula (\ref{leq1}) holds.
\end{proof}

Now, we have the following result.
\begin{thm}\label{AR}The space $T(\II)$ is an AR and homopoty dense in $\overline{T(\II)}$. Therefore, $\overline{T(\II)}$ is also an AR.\end{thm}
\begin{proof}
It follows from Theorem \ref{extension 2} and Lemma \ref{extension} that the space $T(\II)$ satisfies Statement E. By the results  noted in Section 2, we have that
$T(\II)$ is an AR and homopoty dense in $\overline{T(\II)}$. Moreover,  $\overline{T(\II)}$ is also an AR.
\end{proof}

Let $T^{PL}(\II)$ and  $T^{PM}(\II)$ be the sets of all piecewise linear transitive maps and of all piecewise monotone transitive maps. Then, we have
\begin{cor}
The spaces  $T^{PL}(\II)$ and  $T^{PM}(\II)$ are AR's and homopotp dense in both $T(\II)$ and $\overline{T(\II)}$.
\end{cor}\label{cor 1}
\begin{proof}
In \cite[Lemma 2.3]{Kolyada-2151}, see Section 2, Kolyada et al have proved that, using our  terminology, $T^{PL}(\II)$ is homotopy dense in $T(\II)$. Hence  $T^{PM}(\II)\supset T^{PL}(\II)$
is also homotopy dense in $T(\II)$. Therefore, they are AR's.\end{proof}

Secondly, we show that $\overline{T(\II)}$ has SDAP.

\begin{lem}\label{H's properties} The homotopy $H^\gamma:C(\II)\times\II\to C(\II)$ has the following properties:
\begin{enumerate}
\renewcommand{\labelenumi}{(\roman{enumi})}

\item $H^\gamma(\overline{T(\II)}\times\II)\subset \overline{T(\II)}$;

\item For every continuous map $\varepsilon:C(\II)\to (0,1]$, there exists a continuous map $\delta:C(\II)\to (0,1]$
such that
$$d(f,H^\gamma(f,t))<\varepsilon(f)$$
for every $f\in C(\II)$, $\gamma\in [20,+\infty)$ and $t\in [0,\delta(f)).$
\end{enumerate}
\end{lem}
\begin{proof}(i) follows \cite[Lemma 2.3]{Kolyada-2151}. Now we show (ii). We at first show this fact locally holds, that is, for every $f\in C(\II)$, there exists $t(f)\in (0,1)$ and a neighborhood $U(f)$ of $f$
such that
$$d(g,H^\gamma(g,t))<\varepsilon(g)$$
for every $g\in U(f)$, $\gamma\in [20,+\infty)$ and $t\in [0,t(f)].$

Let $\eta(f)=\frac{\varepsilon(f)}{28}.$ Choose a neighborhood $V(f)$ of $f$ such that $\varepsilon(g)\geq\frac{27}{28}\varepsilon(f)$ for every $g\in V(f)$.
Since $f$ is uniformly continuous, there exists $t(f)\in (0,  \eta(f))$ such that
$$|x_1-x_2|\leq t(f)\ \ \ \mbox{implies}\ \ \ |f(x_1)-f(x_2)|<\eta (f).$$
Now let $U(f)=V(f)\cap B(f,\eta(f)).$ Then $U(f)$ and $t(f)$ are as required.
In fact, for every $g\in U(f)$, if $|x_1-x_2|\leq t(f)$, then
$$\begin{array}{rl}
& |g(x_1)-g(x_2)|\\
\leq & |g(x_1)-f(x_1)|+|f(x_1)-f(x_2)|+|f(x_2)-g(x_2)|\\
<&\eta(f)+\eta(f)+\eta(f)\\
=&3\eta(f).\end{array}
$$
Moreover, for every $t\in (0,t(f)]$ and $i\leq s(t)+1$,
$|I^{g,t}_i|\leq t(f)<\eta (f)$ and hence $|g(I^{g,t}_i)|<3\eta(f).$ Thus,
$$\begin{array}{rl}
& a^{i,g,t}_{t}-a^{i,g,t}_{b}\\
\leq & \max g(I^{g,t}_i)-\min g(I^{g,t}_i)+8\alpha^{g,t}_i\\
<& 3\eta (f)+8\max\{\eta(f), 3\eta(f)\}\\
=&27 \eta (f).\end{array}
$$
Hence, $$d(g,H^\gamma(g,t))<27\eta (f)=\frac{27}{28}\varepsilon(f)\leq \varepsilon(g)$$
for every $g\in U(f)$, $\gamma\in [20,+\infty)$ and $t\in [0,t(f)].$

Now consider a locally finite partition of unity  $\{\phi_s:C(\II)\to\II:s\in S\}$ on $C(\II)$ which is subordinated the open cover $\mathcal{U}=\{U(f):f\in C(\II)\}$. For every $s\in S$, choose
$f_s\in C(\II)$ such that $\phi^{-1}_s(
(0,1])\subset U(f_s)$. Let
$$\delta(f)=\sum_{s\in S}\phi_s(f)t(f_s).$$
Then $\delta:C(\II)\to \II$ is as required.

Trivially, $\delta:C(\II)\to (0,1]$ is continuous. Moreover,  for every $f\in C(\II)$ and $t\in [0,\delta(f)]$, let
$$\{s\in S: \phi_s(f)\not=0\}=\{s_1,s_2,\cdots,s_n\}.$$
Then $\delta(f)=\sum_{i=1}^n\phi_{s_i}(f)t(f_{s_i}).$
 Hence, there exists $i\leq n$ such that $\delta(f)\leq t(f_{s_i}).$ Since $f\in \phi^{-1}_{s_i}((0,1])\subset U(f_{s_i})$, using the local statement at $f_i$, we have
 $$d(f,H^\gamma(f,t))< \varepsilon(f)$$
for every $\gamma\in [20,+\infty)$ and $t\in [0,\delta(f)]$.
 \end{proof}

Using the above lemma, we have the following theorem.
\begin{thm}\label{SDAP} The space $\overline{T(\II)}$ has SDAP. Hence, $T(\II)$, $T^{PL}(\II)$ and  $T^{PM}(\II)$  have SDAP.\end{thm}
\begin{proof}
Let $\varepsilon:\overline{T(\II)}\to (0,1]$ and $\phi:K\times\N\to \overline{T(\II)}$ be continuous, where $K$ is a compact metric space.
 By Lemma \ref{H's properties}(ii), there exists a continuous map $\delta:\overline{T(\II)}\to (0,1]$
such that
$$d(f,H^\gamma(f,t))<\frac{\varepsilon(f)}{2}$$
for every $f\in \overline{T(\II)}$, $\gamma\in (20,+\infty)$ and $t\in [0,2\delta(f)).$
Now, define $\gamma_n=20+n$ for every $n$ and
$$\psi(k,n)=H^{\gamma_n}(\phi(k,n),\delta(\phi(k,n)).$$
From Lemma \ref{H's properties}(i) it follows that  $\psi:K\times\N\to \overline{T(\II)}$ is continuous. Now we show that it with the following properties:
\begin{enumerate}
\renewcommand{\labelenumi}{(\roman{enumi})}
\item $d(\phi(k,n),\psi(k,n))<\varepsilon(\phi(k,n))$ for every $(k,n)\in K\times\N$;

\item $\{\psi(K\times\{n\}):n\in\N\}$ is locally finite in $\overline{T(\II)}$.
\end{enumerate}
(i) is trivial. We verify (ii). Otherwise, without loss of generality, there exists $k_n\in K$ for every $n$ such that
$\psi(k_n,n)\to f$ for some $f\in \overline{T(\II)}$. Furthermore, we can assume that $\lim_{n\to\infty} \delta(\phi(k_n,n))=l$ exists.
We consider the following cases:

Case A: $l=0$. We verify that $\phi(k_n,n)\to f$. In fact, for every $\zeta>0$, choose $\eta>0$
such that
$$|x_1-x_2|<\eta\ \mbox{implies} \ |f(x_1)-f(x_2)|<\frac{\zeta}{2}.$$
Choose $N\in\N$ such that $d(f,\psi(k_n,n))<\frac{\zeta}{2}$ and $\delta (\phi(k_n,n))<\eta$ for every $n>N$. Let $I^n_i=I^{\delta (\phi(k_n,n))}_i$ for every $i\leq s(\delta (\phi(k_n,n)))+1.$
Then, for every $x\in\II$ and $n>N$, choose $I^n_i\ni x$. Then $\phi(k_n,n)(x)\leq a^{n,i}_t=\psi(k_n,n)(x')$
for some $x'\in I^n_i$. Then $|x-x'|<\eta$ and $\psi(k_n,n)(x')<f(x')+\frac{\zeta}{2}$. It follows that
$$\phi(k_n,n)(x)<f(x')+\frac{\zeta}{2}<f(x)+\zeta.$$
As the same as,
$$\phi(k_n,n)(x)>f(x)-\zeta.$$
Hence, $d(\phi(k_n,n),f)<\zeta$ for $n>N$.

But, using the continuity of $\delta$, we have
$$0<\delta(f)=\lim_{n\to \infty}\delta(\phi(k_n,n))=0.$$
A contraction occurs.

Case B. $ l>0$. Then there exists $t_0\in (0,1]$ such that $\delta(f)>t_0$ and
$\delta (\phi(k_n,n))>t_0$ for every $n$. Hence, by the definition of $\psi$, for every $t\in [0,t_0]$,
the amplitude of $ \psi(k_n,n)$ in $[0,t]$ is larger than $t_0$ for large enough $n$. Hence it is impossible that  the sequence $ \{\psi(k_n,n)\}$ converges to
a continuous map $f$, which contracts with our assumption!

We have proved that $\overline{T(\II)}$ has SDAP. Moreover,  Theorem \ref{AR} and Corollary \ref{cor 1} show that  $T(\II)$, $T^{PL}(\II)$ and  $T^{PM}(\II)$  are AR's and homotopy dense subspaces in $\overline{T(\II)}$. Hence they have also SDAP.
\end{proof}

At last, we have the following simple theorem.
\begin{thm}\label{complete} The spaces $T(\II)$ and $\overline{T(\II)}$ are topologically complete. \end{thm}
\begin{proof}Trivially, $\overline{T(\II)}$ is topologically complete since it is closed in the completed metric space $C(\II)$. To verify that $T(\II)$ is  topologically complete, it suffices to check that $T(\II)$ is a $G_{\delta}$-set in $C(\II)$.
Choose a countable base $\mathcal{B}$ of $\II$ such that $\emptyset\not\in\mathcal{B}$. For any $(U,V,n)\in\mathcal{B}\times\mathcal{B}\times\N$, let
$$K(U,V,n)=\{f\in C(\II):U\cap f^{-n}(V)\not=\emptyset\}.$$
It is easy to show that $K(U,V,n)$ is open in $C(\II)$ and
$$T(\II)=\bigcap_{(U,V)\in \mathcal{B}\times\mathcal{B}}\bigcup_{n\in\N}K(U,V,n).$$
Thus, $T(\II)$ is a $G_{\delta}$-set in $C(\II)$.
\end{proof}
\begin{proof}[Proof of the Main Theorem] By Theorems \ref{AR}, \ref{SDAP} and \ref{complete}, $T(\II)$ and  $\overline{T(\II)}$ are separable, topologically complete AR's with SDAP. Hence, using  Toru\'{n}czyk's Characterization Theorem of
$\ell_2$, we have that $ T(\II)\approx\overline{T(\II)}\approx \ell_2$.\end{proof}

Let $S(\II)$ be set of all surjective continuous maps from $\II$ onto itself.

\begin{proof}[Proof of Corollary \ref{Z-set}]
 It is not hard to verify that $S(\II)$ is a Z-set in $C(\II)$ and hence $\overline{T(\II)}\subset S(\II)$ is also a Z-set. Therefore, using Main Theorem, Corollary 1 holds.\end{proof}
\begin{proof}[Proof of Corollary \ref{near-homeomorphism}] By Main Theorem, $\overline{T(\II)}\approx \ell_2$. It follows from Theorems \ref{AR} and \ref{complete} that $T(\II)$ is a homotopy dense G$_\delta$-set in
$\overline{T(\II)}$. Using \cite[Theorem 3.4.4]{Sakai-book-2019}, we have this corollary holds. \end{proof}

\section{Remarks, Examples and Open Problems}

It is possible known that $S(\II)\approx \ell_2.$ To sake completeness, we give a proof for this fact.

\begin{thm}We have
 $$S(\II)\approx \ell_2.$$\end{thm}
\begin{proof}Let $\triangle=\{(a,b)\in\II^2: a<b\}$. For $(a,b)\in\triangle$,  define a homeomorphism $h_{(a,b)}:[a,b]\to \II$ as follows
$$h_{(a,b)}(x)=\frac{x-a}{b-a}.$$
Using it, we can define a homeomorphism $H:C(\II)\setminus Const(\II)  \to \triangle\times S(\II)$ as
$$H(f)=((\min f,\max f),h_{(\min f,\max f)}\circ f),$$
where $Const(\II)$ is the set of constant maps from $\II$ to $\II$.
Hence $$\triangle\times S(\II)\approx C(\II)\setminus Const(\II)\approx C(\II)\approx \ell_2$$
since $Cont(\II)\approx \II$ is compact. The last two homeomorphisms can be found in \cite[]{van-book-1989} and \cite{Fan-2018}. Note the following well known fact:
If the product $X\times K$ of a space $X$ and a locally  compact space $K$ is homeomorphic to $\ell_2$, then $X\approx \ell_2.$ See, for example, \cite[Exercise 1.3.9]{Banakh-book}.
It follows from $\triangle$ being locally compact that $S(\II)\approx \ell_2$.
\end{proof}

\begin{rem} The following examples show that $T(\II)\varsubsetneq \overline{T(\II)}\varsubsetneq S(\II)$.
\end{rem}

\begin{ex}\label{ex1}{\rm
Let $f(x)=x^2, x\in \II$.	
Then $f\in S(\II)$.
 We show that $f\not\in \overline{T(\II)}$.
Let $g\in B(f,\frac{1}{100})$.
Note that $f([0,\frac{1}{3}])=[0,\frac{1}{9}]$.
So $g([0,\frac{1}{3}])\subset [0,\frac{1}{3}]$.
This $g$ has a closed invariant proper subset with non-empty interior.
Then $g$ is not transitive.}
\end{ex}

\begin{ex}\label{ex2}{\rm
For $n\geq 5$, we define a map $f_n$ as follows:
\begin{enumerate}
\renewcommand{\labelenumi}{(\roman{enumi})}
	\item for each $k=0,1,\dotsc,n$, $\frac{k}{n}$ is a fixed point of $f$;
	\item $f$ is piecewise linear;
	\item at every monotone subinterval, the slope is great than $3$ and less than $5$.
	\item $[\frac{k-1}{n}, \frac{k+2}{n}]\cap[0,1] \subset f([\frac{k}{n}, \frac{k+1}{n}])\subset [\frac{k-1}{n}, \frac{k+2}{n}]$.
\end{enumerate}
By \cite[Lemma 2.10]{Ruette}, $f_n$ is transitive for every $n$.
It is clear that
$\{f_n\}$ converges to the identity map which is not transitive.}
\end{ex}

For Example \ref{ex1}, we have the following stronger result:
\begin{thm}\label{nowhere-dense}
We have that $\overline{T(\II)}$ is nowhere dense in $S(\II)$.\end{thm}
\begin{proof} Fix any $g\in S(\II)$ and $\varepsilon>0$.
Then there exists a fixed point $x_0\in \II$ for $g$.
Without loss of generality, we assume that $0<x_0<1$.
If $x_0$ is one of the endpoints, we only need to consider the one-sided
neighborhood.
There exists $\delta \in (0, \min\{\frac{\varepsilon}{4}, x_0,1-x_0\})$ such that $|g(x)-x_0|<\frac{\varepsilon}{4}$ for any $x\in \II$ with $|x-x_0|<\delta$.
We define a map $h$ as follows:
\begin{enumerate}
\renewcommand{\labelenumi}{(\roman{enumi})}
	\item if $x\not\in (x_0-\delta,x_0+\delta)$, $h(x)=g(x)$;
	\item $h(x_0-\frac{\delta}{2})=x_0-\frac{\delta}{2}$
	and $h$ is linear on $[x_0-\delta,x_0-\frac{\delta}{2}]$;
	\item $h(x_0+\frac{\delta}{2})=x_0+\frac{\delta}{2}$
	and $h$ is linear on $[x_0+\frac{\delta}{2},x_0+\delta]$;
	\item if $x\in [x_0-\frac{\delta}{2},x_0+\frac{\delta}{2}]$,
	$h(x)=\frac{1}{\delta}(x-x_0+\frac{\delta}{2})^2+ x_0-\frac{\delta}{2}$.
\end{enumerate}
Then $d(g,h)<\varepsilon$. Moreover, as in Example \ref{ex1}, we can verify that $h\in S(\II)\setminus \overline{T(\II)}.$  \end{proof}

We guess that $\overline{T(\II)}$ is a Z-set in $S(\II)$. Moreover, we guess that the following problem has a positive answer:
\begin{prob} Does the following four-homeomorphism hold?
$$(C(\II), S(\II), \overline{T(\II)},T(\II))\approx $$$$ (\II\times \II\times Q\times \ell_2, \{0\}\times \II\times Q\times \ell_2,\{0\}\times \{0\}\times Q\times \ell_2,\{0\}\times \{0\}\times s\times \ell_2),$$
where $Q=[-1,1]^\N,s=(-1,1)^\N.$\end{prob}

\begin{rem}Although the spaces $T^{PL}(\II)$ and $T^{PM}(\II)$ are AR's with SDAP, non of them is homeomorphic to $\ell_2$ since they are Z$_{\sigma}$-spaces. In fact, let
$PL_n(\II)$ be the set of all piecewise linear continuous maps of modality $\le n$. It is an elementary exercise to verify that $PL_n(\II)$ is closed in $C(\II)$. Hence $PL_n(\II)\cap T^{PL}(\II)$ is closed in $T^{PL}(\II)$. Moreover, the homotopy $H^\gamma$ shows that $PL_n(\II)\cap T^{PL}(\II)$ is homotopy negligible in $T^{PL}(\II)$ when $\gamma$ is large enough. Hence $PL_n(\II)\cap T^{PL}(\II)$ is a Z-set in $T^{PL}(\II)$. Thus,
$T^{PL}(\II)=\bigcup_{n=1}^\infty PL_n(\II)\cap T^{PL}(\II)$ is a Z$_{\sigma}$-space. Similarly,  $T^{PM}(\II)$ is also a Z$_{\sigma}$-space.
\end{rem}

\begin{rem} In Proof of Lemma \ref{extension}, we use convex composition of transitive maps. But, neither $T(\II)$ nor $\overline{T(\II)}$ is not convex. In fact, it is easy to choose a transitive map $f$ such that $1-f$ is also transitive. But $\frac{1}{2}f+\frac{1}{2}(1-f)=\frac{1}{2}$ is not in $\overline{T(\II)}$.
\end{rem}

 \end{document}